\documentclass[11pt]{amsart}
\usepackage{amssymb, latexsym, graphicx, mathrsfs, enumerate}
\usepackage{xcolor}
\definecolor{forestgreen}{rgb}{0.0, 0.27, 0.13}
\definecolor{ultramarine}{rgb}{0.07, 0.04, 0.56}
\definecolor{lust}{rgb}{0.9, 0.13, 0.13}

\usepackage[colorlinks=true,citecolor=teal, linkcolor=purple, urlcolor=violet, filecolor=green, backref=page]{hyperref}
\usepackage{hyperref}
\usepackage{indentfirst}
\usepackage{changes}
\definechangesauthor[name={KY}, color=lust]{KY}

\usepackage{manfnt} %called for \newcommand*\cube{\mbox{\mancube}}
\usepackage{amssymb,amsmath}
\makeindex

\newcommand{\cD}{\mathcal{D}}
\newcommand{\cE}{\mathcal{E}}

\newcommand{\cLC}{\mathcal{LC}}
\newcommand{\cM}{\mathcal{M}}
\newcommand{\cO}{\mathcal{O}}
\newcommand{\cS}{\mathcal{S}}

\newcommand{\bQ}{\mathbb{Q}}

\newcommand{\bZ}{\mathbb{Z}}

\DeclareFontFamily{U}{wncy}{}
\DeclareFontShape{U}{wncy}{m}{n}{<->wncyr10}{}
\DeclareSymbolFont{mcy}{U}{wncy}{m}{n}
\DeclareMathSymbol{\Sha}{\mathord}{mcy}{"58}

\newcommand{\lp}{\left(}
\newcommand{\rp}{\right)}

\newcommand{\lbrb}[1]{\lp #1 \rp}
\newcommand{\lcrc}[1]{\{ #1 \}}
\newcommand{\lfrf}[1]{\lfloor #1 \rfloor}
\newcommand{\labrab}[1]{\left| #1 \right|}

\DeclareMathOperator{\Sel}{Sel}

\theoremstyle{plain}% default style
\newtheorem{theorem}{Theorem}[section]
\newtheorem{lemma}[theorem]{Lemma}
\newtheorem{proposition}[theorem]{Proposition}
\newtheorem{corollary}[theorem]{Corollary}

\theoremstyle{definition} % definition style

\theoremstyle{remark} % remark style
\newtheorem*{remark}{Remark}

\numberwithin{equation}{section}

\makeatletter
\newcommand{\alignfootnote}[1]{%
	\ifmeasuring@
	\else
	\footnote{#1}%
	\fi
}
\makeatother
\begin{document}
\sloppy %Prevents long lines to go into the right margin
\title{On the distribution of analytic ranks of elliptic curves}

\author[Peter Cho]{Peter J. Cho$^{\dagger}$}
\thanks{$^{\dagger}$ This work was supported by Basic Science Research Program through the National Research Foundation of Korea(NRF) funded by the Ministry of Education(2019R1F1A1062599).}
\address{Department of Mathematical Sciences, Ulsan National Institute of Science and Technology, UNIST-gil 50, Ulsan 44919, Republic of  Korea}
\email{petercho@unist.ac.kr}

\author[Keunyoung Jeong]{Keunyoung Jeong$^{\star}$}
\thanks{$^{\star}$ The author was supported by the National Research Foundation of Korea(NRF) funded by the Ministry of Education, under the Basic Science Research Program(2019R1C1C1004264).}
\address{Department of Mathematical Sciences, Ulsan National Institute of Science and Technology, UNIST-gil 50, Ulsan 44919, Republic of  Korea}
\email{kyjeong@unist.ac.kr}

\subjclass[2010]{Primary 11G05,  Secondary 11G40, 11M26 }

\keywords{elliptic curve, rank, trace formula, GRH, BSD conjecture}
\begin{abstract} 
In this paper, under GRH for elliptic $L$-functions, we give an upper bound for the probability for  an elliptic curve with analytic rank $\leq a$ for $a \geq 11$, and also give an upper bound of $n$-th moments of analytic ranks of elliptic curves. These are applications of counting elliptic curves with local conditions, for example, having good reduction at $p$.  
%We show that, under GRH for elliptic $L$-functions, the proportion of elliptic curves with analytic ranks $r_E \leq 11$ (resp. $r_E \leq 21$) is at least $0.935185$ (resp. $0.996548$). We also give an explicit upper bound on  the limit sup of  $n$-th moment of analytic ranks of elliptic curves for every positive integer $n$. These theorems are partial results towards some implications of the recent conjecture  by Park, Poonen, Voight, and Wood \cite{PPVW}. We prove the main theorems by using a sort of trace formula for elliptic curves, and the trace formula is established by counting elliptic curves with a finite number of local conditions. The main theorems may be understood as the $n$-level density with multiplicity for the family of elliptic $L$-functions.  
\end{abstract}

\maketitle

\section{Introduction} \label{Intro}
In this article, we study the distribution of analytic ranks of elliptic curves over the rational numbers.
Let $P(r_E \geq a)$ be the proportion of elliptic curves with analytic rank $r_E \geq a$. We give a numerical bound for $P(r_E \geq a)$.

\begin{theorem} \label{rank-dist}
Assume GRH for elliptic $L$-functions.  Let $C$ be a positive constant, let $n$ a positive integer. We have
\begin{align*}
P\left(r_E \geq (1+C) 9n \right)  \leq \frac{\sum_{k=0}^{n}{ {2n} \choose {2k}}\left( \frac 12\right)^{2n-2k}(2k)!\left( \frac 16 \right)^k }{(C \cdot 9n)^{2n}}.
\end{align*} 
\end{theorem}

Now, we can give a numerical evidence that there are at most a small proportion of elliptic curves with large analytic ranks.
\begin{corollary}
 Assume GRH for elliptic $L$-functions.  Let $f(t)=\sum_{k=0}^{t}{ {2t} \choose {2k}}\left( \frac 12\right)^{2t-2k}(2k)!\left( \frac 16 \right)^k$. Then,
\begin{enumerate}
\item $P(r_E \leq a ) \geq 1-\frac{f(1)}{(a-8)^2}$ for $11 \leq a \leq 17$.
\item $P(r_E \leq a) \geq 1-\min_{1 \leq l \leq n} \left\{ \frac{f(l)}{(a+1-9l)^{2l}}\right\}$ for $9n \leq a \leq 9n+8$, $n \geq 2$.
\end{enumerate}
For small $a$'s, $P(r_E \leq a)$ is recorded in the following table. 

\begin{center}
\begin{tabular}{|c|c|c|c|c|c|c|c|}
\cline{1-8}
 $a$ &  $P(r_E \leq a)$ &  $a$ &  $P(r_E \leq a)$ & $a$ & $P(r_E \leq a)$  & $a$ & $P(r_E \leq a)$ \\ \cline{1-8}
 $11$ &  $ \geq 0.935185$ & $16$ & $\geq 0.990885$ &  $21$  & $\geq 0.996548$    & $26$   & $\geq 0.999812$\\ \cline{1-8}
 $12$ &   $\geq  0.963541$ & $17$ &  $\geq 0.992798$& $22$ & $\geq 0.998033$ &  $27$ &$\geq 0.999877$ \\ \cline{1-8}
 $13$ &  $\geq 0.976666$ & $18$ &  $\geq 0.994166$  &  $23$  & $\geq 0.999051$ &  $28$ & $\geq 0.999916$\\ \cline{1-8}
$14$ &  $\geq 0.983796$ & $19$ &  $\geq 0.995179$ &  $24$ & $\geq 0.999488$  &  $34$ &$\geq 0.999985$\\ \cline{1-8}
$15$ & $\geq 0.988095$ & $20$ &  $\geq 0.995949$ &  $25$  & $\geq 0.999699$  & $35$  &$\geq 0.999988$ \\ \cline{1-8}
\end{tabular}
\end{center}
\end{corollary}

\begin{remark}
(1) The model of \cite{PPVW} conjectures
that there are only finitely many elliptic curves with rank $> 21$,
and we show that the proportion of elliptic curves with analytic rank $> 21$  is less than 
0.0035 under the GRH for elliptic $L$-functions.

(2) Heath-Brown's result \cite[Theorem 2]{Hea} seems stronger than ours for a very large rank $a$. However, our emphasis is on a small rank $a$ not a large rank $a$. Theorem \ref{rank-dist} is not an explicitation of  the implicit constant in \cite[Theorem 2] {Hea}, and our method of the proof is different. For the proof, we establish the Frobenius trace formula for elliptic curves (Theorem \ref{traceF}), which is a new tool for elliptic curves. In the process of proof, we encounter the following inequality:
\begin{align*}
\left| \{ E \in \cE(X) | r_E \geq  \frac{1+C}{\sigma_{2n}} \} \right| & \left( \frac{C\sigma_{2n}}{4} \right)^{2n} 
 \leq \sum_{E \in \cE(X)}  \left( -\frac{2}{\log X} \sum_{m_i}\frac{\widehat{a}_E(m_i)\Lambda(m_i)}{\sqrt{m_i}}\widehat{\phi}_{2n} \left( \frac{\log m_i}{\log X}\right)  \right)^{2n}.
\end{align*}
Since, using the Frobenius trace formula, we have the equality
\begin{align*}
& \sum_{E \in \cE(X)}  \left( -\frac{2}{\log X} \sum_{m_i}\frac{\widehat{a}_E(m_i)\Lambda(m_i)}{\sqrt{m_i}}\widehat{\phi}_{2n} \left( \frac{\log m_i}{\log X}\right)  \right)^{2n}\\ &=  \left( \frac{\sigma_{2n}^2}{4}\right)^{2n} \sum_{S_2 \subset \{1,2,3,\dots,2n \}} \left( \frac{1}{2} \right)^{|S_2^c|}|S_2|!\left( \frac 16 \right)^{|S_2|/2}|\cE(X)| + O\left(\frac{X^{\frac 56} }{\log X} \right),
\end{align*}
we don't lose any information for the inequality. See \cite[Sec. 7]{Hea}. Our approach is in the same spirit with Katz and Sarnak's $n$-level density conjecture. For the introduction to the $n$-level density conjecture and some partial results, we refer to \cite{Rub, Mil, CK15}. 

(3) 
In the proof of the Frobenius trace formula, Theorem \ref{traceF}, we use similar arguments in \cite{CK15}. The new feature here is the Eichler-Selberg trace formula which replaces the role of orthogonality of characters in \cite{CK15}. 

(4) For algebraic ranks, we remark that the recent development of \cite{BS15, BS} gives better bounds than ours. 
Since the average of the order of $\Sel_5(E/\bQ)$ is 6 by \cite{BS}, we have
$$  P(r_E \geq a) 5^a \leq 6$$
by the exact sequence
$$0 \to \frac{E(\bQ)}{5E(\bQ)} \to \Sel_5(E/\bQ) \to \Sha(E/\bQ)[5] \to 0.$$
Hence for example, $P(r_E < 20)$ is bounded by $1 - 6/5^{20} \approx 0.99999999999993708544$.
However, for the average of analytic ranks, Young's bound $\frac{25}{14}$ \cite{Y} under GRH is the best record. It gives a  bound for $P(r_E \geq a)$, which is $\frac{25}{14a}$.
\end{remark}

The second one is regarding the moments of the analytic ranks of elliptic curves.
One can show that the limsup of $n$-th moments of analytic ranks of elliptic curves exist, by applying the result of Heath-Brown \cite[Theorem 2]{Hea}.
Using an approach of Miller \cite{Mil}, we propose  an explicit upper bound on  the $n$-th moment of analytic ranks for every positive integer $n$.

\begin{theorem} \label{n-th moment} Assume GRH for elliptic $L$-functions. Let $r_E$ be the analytic rank of an elliptic curve $E$. Let $\sigma_n=\frac{2}{9n}$ for a positive integer $n$. 
For every positive integer $n$, we have
\begin{align*}
\limsup_{X \rightarrow \infty} \frac{1}{|\cE(X)|}\sum_{E \in \cE(X)}r_E^n \leq \sum_{S}\left( \frac{1}{\sigma_n} \right)^{|S^c|}\sum_{\substack{S_2 \subset S  \\ |S_2| \text{even}}} \left( \frac 12  \right)^{|S_2^c|} \left| S_2 \right|! \left(\frac 16\right)^{|S_2|/2}. 
\end{align*}
where $S$ runs over subsets of $\{1,2,3,\dots,n\}$, and $S_2$ runs over subsets of even cardinality of the set $S$.
In particular, the limit sup of the 2nd moment of analytic ranks is bounded by $90.584$ and the limit sup of the 3rd moment of analytic ranks is bounded by $2758$. 
\end{theorem}

%In \cite{Bru}, the author shows that the number of elliptic curves whose height is less than $X$ is $\frac{4}{\zeta(10)}X^{\frac{5}{6}} + O(X^{\frac{1}{2}})$. In this paper, we will compute the number of elliptic curves with a given local condition at prime $p \geq 5$.

These results are obtained by accurate count of elliptic curves with local conditions,
which is of its own interest.
Here a local condition at $p$ means one of good reduction, bad reduction, multiplicative reduction, additive reduction, $a_p$ in the Weil bound\footnote{An elliptic curve $E$ satisfies a local condition $a_p$ if $a_p(E)=a_p$.},   split and non-split reduction,  or one of the Kodaira--N\'eron type.

Before stating the results, we explain the model of elliptic curves in our consideration.
We treat elliptic curves of the form: 
\begin{equation*} \label{model}
E_{A,B} : y^2 = x^3 + Ax + B,\,\, A, B \in \bZ \qquad \textrm{if a prime $q$ satisfies $q^4 \mid A$, then $q^6 \nmid B$}.
\end{equation*}
We denote the prime condition by $(M)$. Elliptic curves in this model might be  not minimal at $2$ and $3$ but are minimal at all primes $q \geq 5$. 
Also, each isomorphism class of elliptic curves appears in this model exactly one time.
We define a naive height of $E_{A,B}$ by 
$$H(E_{A,B}):=\max(|A|^3, |B|^2).$$ 
The family of our interest is 
\begin{equation*} \label{defcE}
\cE(X) :=  \lcrc{(A, B) \in \bZ^2 :  4A^3 + 27B^2 \neq 0, H(E_{A,B}) \leq X, \textrm{  and $(A, B)$ satisfies $(M)$}},
\end{equation*}
and let
\begin{align*}
&\cE_p^{\cLC}(X)=\{(A, B) \in \cE(X) : E_{A,B} \textrm{ satisfies $\cLC$ at the prime } p\},
\end{align*}  
where $\cLC$ is one of the local conditions.
Then, we have
\begin{theorem} \label{main1}
For $5 \leq p \leq X^{\frac{1}{3m_\cLC} }$,
\begin{align*}
&\left| \cE_p^{\cLC}(X) \right|=c_\cLC(p)\frac{4}{\zeta(10)}X^{\frac 56}+ O(h_{\cLC,p}(X)),
\end{align*}
where $m_\cLC, c_\cLC(p), h_{\cLC,p}(X)$ are given by the following table.
\end{theorem}
\begin{center}
\begin{tabular}{|c|c|c|c|}
\cline{1-4}
Local conditions & $c_\cLC(p)$ & $h_{\cLC,p}(X)$ & $m_\cLC$ \\ \cline{1-4}
 good &  $ \frac{p^2-p}{p^2} \frac{p^{10}}{p^{10}-1}$ & $pX^{\frac{1}{2}}$ & $1$  \\ \cline{1-4}
 bad &  $ \frac{p^{10}-p}{p^{11}} \frac{p^{10}}{p^{10}-1}$ & $pX^{\frac{1}{2}}$ & $1$  \\ \cline{1-4}
 mult &  $ \frac{p-1}{p^2} \frac{p^{10}}{p^{10}-1}$ & $X^{\frac{1}{2}}$ & $1$  \\ \cline{1-4}
  split, non-split &  $ \frac{p-1}{2p^2} \frac{p^{10}}{p^{10}-1}$ & $X^{\frac{1}{2}}$ & $1$  \\ \cline{1-4}
  addi &  $ \frac{p^9-p}{p^{11}} \frac{p^{10}}{p^{10}-1}$ & $pX^{\frac{1}{2}}$ & $1$  \\ \cline{1-4}
  $a$ &  $ \frac{(p-1)H(a^2-4p)}{2p^2} \frac{p^{10}}{p^{10}-1}$ & $ H(a^2-4p)X^{\frac{1}{2}}$ & $1$  \\ \cline{1-4}
 $\mathrm{I}_m$&  $\frac{1}{p^m}\left( 1-\frac 1p \right)^2 \frac{p^{10}}{p^{10}-1}$ & $pX^{\frac{1}{2}}$ & $m+1$  \\ \cline{1-4}
 $\mathrm{II}$&   $\frac{1}{p^2}\left( 1-\frac 1p \right) \frac{p^{10}}{p^{10}-1}$ &$X^{\frac{1}{2}}$ & $2$ \\ \cline{1-4}
 $\mathrm{III}$&  $\frac{1}{p^3}\left( 1-\frac 1p \right) \frac{p^{10}}{p^{10}-1}$ &$X^{\frac{1}{2}}$ & $3$ \\ \cline{1-4}
$\mathrm{IV}$&  $\frac{1}{p^4}\left( 1-\frac 1p \right) \frac{p^{10}}{p^{10}-1}$ &$X^{\frac{1}{2}}$ & $4$ \\ \cline{1-4}
$\mathrm{IV}^*$&$\frac{1}{p^7}\left( 1-\frac 1p \right) \frac{p^{10}}{p^{10}-1}$ &$p^{-1}X^{\frac{1}{2}}$ & $6$   \\ \cline{1-4}
$\mathrm{III}^*$&$\frac{1}{p^8}\left( 1-\frac 1p \right) \frac{p^{10}}{p^{10}-1}$  &$p^{-3}X^{\frac{1}{2}}$ & $5$ \\ \cline{1-4}
$\mathrm{II}^*$& $\frac{1}{p^9}\left( 1-\frac 1p \right) \frac{p^{10}}{p^{10}-1}$  &$p^{-2}X^{\frac{1}{2}}$ & $7$ \\ \cline{1-4}
$\mathrm{I}_m^*$& $\frac{1}{p^{m+5}}\left( 1-\frac 1p \right)^2 \frac{p^{10}}{p^{10}-1}$ &$pX^{\frac{1}{2}}$ & $m+6$ \\ \cline{1-4}
$\mathrm{I}_0^*$&  $\frac{1}{p^5}\left( 1-\frac 1p \right) \frac{p^{10}}{p^{10}-1}$ &$pX^{\frac{1}{2}}$ & $5$ \\ \cline{1-4}
\end{tabular}
\end{center}
Here $H(\cdot)$ is the Hurwitz class number.

We can count elliptic curve not only with a single local condition but also with finitely many local conditions. Let $\cS= \{ \cLC_{p_i} \}$ be a finite set of local conditions, where $\cLC_p$ can be $\textrm{good}$, $\textrm{bad}$, $\textrm{split}$, $\textrm{non-split}$, $\textrm{additive}$, $a_p$, or $T$, one of the Kodaira--N\'eron type. Let $|\cLC_p|$ be the probability given in Theorems \ref{main1} and let us define $\left|\cS \right| =\prod_{i}\left|\cLC_{p_i} \right|$.
We simply denote  $m_{\cLC_{p_i}}$ by $m_i$.
Let
$$
\cE^{\cS}(X)=\left\{  (A,B)\in \cE(X) : E_{A,B} \mbox{ satisfies }\cS  \right\}.
$$

In the following theorem, we show that these conditions are independent.
\begin{theorem} \label{main4}
Let  $\cLC_{i}$ be a local condition at $p_i$ where $p_i$ is a prime greater than $3$ and $\prod_{i=1}^n p_i^{m_i}  \leq X^{\frac{1}{3}-\epsilon}$ for some $\epsilon >0$. Then, we have
\begin{align*}
\cE^{\cS}(X)=|\cS|\frac{4}{\zeta(10)}X^{\frac 56} + O\left( \left( \prod_{i=1}^n c_{p_i} \right) X^{\frac 12}\right),
\end{align*}
where $c_{p_i} = |\cLC_{p_i}| \lbrb{\frac{p_i^{10}-1}{p_i^{10}}}p_i^{m_i}$.
%for some $m_i$ determined by $\cLC_{p_i}$. 
We note that when $\cLC_{p} = a$, then $c_{p} = \frac{p-1}{2p}H(a^2-4p)$.
\end{theorem}

In \cite[\S 3.4]{Wat}, Watkins predicted probabilities for \textrm{Kodaira--N\'eron} types at a prime $p \geq 5$ with a heuristic approach. Our results agree with those in \cite{Wat}. 
Also, we found a result \cite{BCD} of counting elliptic curves with a local condition $a_p$. In \cite{BCD} they consider the family of elliptic curves $y^2=x^3+Ax+B$ without the minimality condition $(M)$, hence it contains many isomorphic elliptic curves. 
Their dominant error term is $O(X^\frac{5}{6}/p)$, which is much weaker than $O(H(a^2-4p)X^{\frac{1}{2}})$ in Theorem \ref{main1} for a small prime $p$.
%It is worth to remark that in \cite{W}, the author used a similar method to show that about 17.9\% of elliptic curves over the rational numbers are semistable.

In Cremona and Sadek's recent preprint \cite{CS}, the authors consider similar problems.
They uses the general Weierstrass model, so their approach gives a natural probabilities for each local conditions at $p=2$ or 3. 
%Although our method also gives a probabilities for $p=2$ or 3, but the result is unexpected, since when deriving the naive height model we divide and multiply 2 and 3 many times.
In their model, they give the probability for semi-stability, which is 60.85\%. This result seems more natural than Wong's probability for semistability 17.9\% \cite{W}.
%Further, they also give a result on the probability on local conditions on infinitely many primes.
Even though we can count elliptic curves with finitely many local conditions, we have an explicit error term. There will be potential applications in elliptic $L$-functions. 

%In this paper, we compute the explicit error terms in Theorem \ref{main1} that leads to the applications.

In Section \ref{counting}, we give a proof for Theorem \ref{main1} and \ref{main4}.
In Section \ref{trace}, we state and give a proof for the Frobenius trace formula for elliptic curves. Section \ref{distribution} is devoted for the proof for Theorem \ref{rank-dist} and \ref{n-th moment}.

\section{Counting Elliptic Curves} \label{counting}

\subsection{With a single local condition}
In this section, we give proofs of Theorems \ref{main1}.
For positive integers $m$ and $n$, we say $\alpha \in \bZ/p^m\bZ$ is divided by $p^n$ if
$n < m$ and $\alpha = a p^n$ for some $a \in \bZ/p^m\bZ$, and exactly divided by $p^n$ if
there exists $u \in (\bZ/p^m\bZ)^\times$ such that $\alpha = u p^n$.
Let
$$\cD_{p^m, \alpha, \beta}(X) := \lcrc{(A,B) \in \cD(X) :  (A, B) \equiv (\alpha, \beta) \pmod{p^m}},$$
and
$$\cM_{p^m, \alpha, \beta}(X) := \lcrc{(A,B) \in \cD_{p^m, \alpha, \beta}(X) : (A,B) \textrm{ satisfies }(M)},$$
for $(\alpha, \beta) \in (\bZ/p^m\bZ)^2$ such that $p^4 \nmid \alpha$ or $p^6 \nmid \beta$.
For nonzero $(\alpha, \beta) \in (\bZ/p^m\bZ)^2$ such that $p^4 \nmid \alpha$ or $p^6 \nmid \beta$, we define
$$\cE_{p^m, \alpha, \beta}(X) := \lcrc{(A,B) \in \cM_{p^m, \alpha, \beta}(X) : 4A^3 + 27B^2 \neq 0},$$
and $\cS_{p^m, \alpha, \beta}(X)$ as the unique set satisfying $\cM_{p^m, \alpha, \beta}(X) =  \cE_{p^m, \alpha, \beta}(X) \bigsqcup \cS_{p^m, \alpha, \beta}(X)$.
%Then, an element in  $\cM_{p^m, \alpha, \beta}(X)$ (resp. $\cS_{p^m, \alpha, \beta}(X)$, $\cE_{p^m, \alpha, \beta}(X)$) represents a curve $y^2 = x^3 + Ax + B$ (resp. singular curve, elliptic curve) whose mod $p^m$ reduction is a curve given by $y^2 = x^3 + \alpha x + \beta$, and $H(E_{A,B}) \leq X$.

\begin{lemma} \label{box}
Let $m \leq C, n \leq D$ and $\alpha \in \bZ/m\bZ, \beta \in \bZ/n\bZ$.
The cardinality of a set
$$\lcrc{(c, d) \in \bZ^2 : |c| \leq C, |d| \leq D, (c, d) \equiv (\alpha, \beta) \pmod{\bZ/m\bZ \times \bZ/n\bZ}}, $$
is
$$4 \cdot \lfrf{\frac{C+1}{m}}\cdot \lfrf{\frac{D+1}{n}} + O\lbrb{ \lfrf{\frac{C+1}{m}} +  \lfrf{\frac{D+1}{n}} +1}.
$$
\end{lemma}

We define $*$-operator on a pair by  $d*(a,b) := (d^4a, d^6b)$.

\begin{lemma} \label{DMmodp}
For a nonzero $(\alpha, \beta) \in (\bZ/p^m\bZ)^2$ such that $p^4 \nmid \alpha$ or $p^6 \nmid \beta$,
$$\cD_{p^m, \alpha, \beta}(X) =  \bigsqcup_{\substack{d \leq X^{\frac{1}{12}} \\ p \nmid d }}d*\cM_{p^m, d^{-4}\alpha, d^{-6}\beta}(d^{-12}X).$$
\end{lemma}
\begin{proof}
This is an analogue of \cite{Bru}, but note that the condition $p^4 \nmid \alpha$ or $p^6 \nmid \beta$ gives a restriction on $d$.
\end{proof}
We note that Lemma \ref{DMmodp} gives
\begin{equation*} \label{eqn:DMmodp} 
|\cM_{p^m, \alpha, \beta}(X)| = \sum_{\substack{ d \leq X^{\frac{1}{12}} \\ p \nmid d}} \mu(d)|\cD_{p^m, d^{-4}\alpha, d^{-6}\beta}(d^{-12}X)|.
\end{equation*}

\begin{proposition} \label{42}
For a prime $5 \leq p  \leq X^{\frac{1}{3m}}$ and nonzero $(\alpha, \beta) \in (\bZ/p^m\bZ)^2$ such that $p^4 \nmid \alpha$ or $p^6 \nmid \beta$,
$$ |\cE_{p^m, \alpha, \beta}(X)| =  \frac{1}{p^{2m}} \frac{p^{10}}{p^{10}-1}\frac{4}{\zeta(10)}X^{\frac{5}{6}} + O\lbrb{\frac{X^{\frac{1}{2}}}{p^m}}.$$
Furthermore,
$$|\cE_{p,0,0}(X)| =  \frac{p^8-1}{p^{10}-1}\frac{4}{\zeta(10)}X^{\frac{5}{6}} + O(pX^{\frac{1}{2}}).$$
\end{proposition}
\begin{proof}
We have
$|\cE_{p^m,\alpha,\beta}(X)| = |\cM_{p^m,\alpha,\beta}(X)| +O( | \cS_{p^m, \alpha, \beta}(X) | )$,
and $|\cS_{p^m, \alpha, \beta}(X)| \ll O(X^{\frac 16} /p^m)$.
On the other hand 
$$|\cD_{p^m, d^{-4}\alpha, d^{-6}\beta}(\frac{X}{d^{12}})| = \frac{4X^{\frac{5}{6}}}{p^{2m}d^{10}} + O\left( \frac{X^{\frac{1}{2}}}{p^{m}d^6} +1 \right)$$
by Lemma \ref{box}.
Together with Lemma \ref{DMmodp} for $(\alpha, \beta)$ such that $p^4 \nmid \alpha$ or $p^6\nmid \beta$,  we have
\begin{align*}
|\cE_{p^m,\alpha,\beta}(X)|&=  |\cM_{p^m,\alpha,\beta}(X)| +O(|\cS_{p^m, \alpha, \beta}(X) | )
= \sum_{\substack{d \leq X^{\frac{1}{12}} \\ p \nmid d}} \mu(d)|\cD_{p^m, d^{-4}\alpha, d^{-6}\beta}(\frac{X}{d^{12}})| +O\lbrb{\frac{X^{\frac{1}{6}}  }{p^m}},
\end{align*}
which is equal to 
\begin{align*}
&=\sum_{\substack{d \leq X^{\frac{1}{12}} \\ p\nmid d }} \mu(d) \lbrb{\frac{4X^{\frac{5}{6}}}{p^{2m}d^{10}} + O\left( \frac{X^{\frac{1}{2}}}{p^md^6}  + 1\right) } +O\lbrb{\frac{X^{\frac 16} }{p^m}} \\
&= \sum_{\substack{d =1 \\ p \nmid d}}^{\infty} \mu(d) \lbrb{\frac{4X^{\frac{5}{6}}}{p^{2m}d^{10}} + O\left( \frac{X^{\frac{1}{2}}}{p^md^6} \right)} +O\lbrb{X^{\frac{1}{12}} + \frac{X^{\frac 16} }{p^m}}  -  \sum_{\substack{d \geq X^{\frac{1}{12}} \\ p \nmid d}} \mu(d) \lbrb{\frac{4X^{\frac{5}{6}}}{p^{2m}d^{10}} + O\left(\frac{X^{\frac{1}{2}}}{p^md^6}  \right)} \\
&= \frac{4X^{\frac{5}{6}}}{p^{2m} } \lbrb{\sum_{\substack{d =1 \\ p \nmid d}}^{\infty} \frac{\mu(d)}{d^{10}}} + O\left(
\frac{X^{\frac{1}{12}}}{p^m}  + \frac{X^{\frac{1}{2}}}{p^m} 
\sum_{d=1}^{\infty}\frac{1}{d^6} + X^{\frac{1}{12}} + \frac{X^{\frac{1}{6}}}{p^m}\right) \\
&= \frac{1}{p^{2m} }\frac{p^{10}}{p^{10}-1} \frac{4}{\zeta(10)}X^{\frac{5}{6}} + O\lbrb{\frac{X^{\frac{1}{2}}}{p^m}}.
\end{align*}
%We note that the error terms in $(*)$ bounded by
%$$\labrab{\sum_{d \geq X^{\frac{1}{12}}} \mu(d) \lbrb{\frac{4X^{\frac{5}{6}}}{p^{2m}d^{10}} + O(1)\frac{X^{\frac{1}{2}}}{p^md^6} + O(1)\frac{X^{\frac{1}{3}}}{p^md^4}}}  \ll \sum_{d \geq X^{\frac{1}{12}}} \lbrb{\frac{4X^{\frac{5}{6}}}{p^{2m}d^{10}} + \frac{X^{\frac{1}{2}}}{p^md^6} + \frac{X^{\frac{1}{3}}}{p^md^4}} \ll \frac{X^{\frac{1}{12}}}{p^m}. $$
Hence we have the first one.
On the other hand, the computation of $|\cE(X)|$ from \cite[Lemma 4.3]{Bru} and 
$$|\cE(X)| = |\cE_{p,0,0}(X)| + \sum_{\alpha, \beta \in (\bZ/p\bZ)^\times} |\cE_{p, \alpha, \beta}(X)|,$$
give the result for $|\cE_{p,0,0}(X)|$.
\end{proof}

We collect auxiliary lemmas on the number of elliptic curves corresponding to each local condition. 
For the good and bad reduction conditions, we use the cardinality of 
\begin{equation*}
\lcrc{(\alpha, \beta) \in (\bZ/p\bZ)^2 : 4\alpha^3 + 27\beta^2 \equiv 0 \pmod{p}},
\end{equation*}
which is $p$. 
Let $K$ be an imaginary quadratic field and $\cO$ be an order of $K$. We define the Hurwitz class number
$H(\cO)$ by 
$$ H(\cO) := \sum_{\cO \subset \cO' \subset \cO_K}\frac{2}{|\cO'^\times|}h(\cO'),$$
where $h(\cO')$ is the class number of an order $\cO'$.
We also write $H(\cO)$ as $H(D)$, when $D$ is the discriminant of $\cO$.
By \cite[Theorem 14.18]{Cox}, we have 
\begin{equation} \label{eqn:Cox}
\labrab{\lcrc{(\alpha, \beta) \in (\bZ/p\bZ)^2 : a_p(E_{\alpha, \beta}) = a} } = \frac{p-1}{2}H(a^2-4p).
\end{equation}
For multiplicative reduction, we need
\begin{lemma} \label{singpn}
Let $p\geq 5$ be a prime and $m \geq 1$ be an integer. Then,
\begin{equation*}
\labrab{ \lcrc{ (\alpha, \beta) \in \lbrb{(\bZ/p^m\bZ)^\times}^2 : 4\alpha^3 + 27\beta^2 \equiv 0 \pmod{p^m}}  } = p^m\lbrb{1- \frac{1}{p}},
\end{equation*}
and
$$\labrab{ \lcrc{ (\alpha, \beta) \in \lbrb{(\bZ/p\bZ)^\times}^2 : 4\alpha^3 + 27\beta^2 \equiv 0 \pmod{p}, \beta = -\beta'^2 } } = \frac{p-1}{2}.$$
\end{lemma}

The following is a result of Tate's algorithm for $E_{A,B}$.

\begin{lemma} \label{Tatealg}
Let $p \geq 5$ and $E_{A,B} : y^2 = x^3 + Ax + B$ be a minimal model over $\bQ_p$. Then, $E$ has
$$
\left\{ \begin{array}{cccc}
\textrm{Type $\mathrm{II}$ if and only if }  p \mid A, p \| B &  \\ 
\textrm{Type $\mathrm{III}$ if and only if } p \| A, p^2 \mid B &  \\ 
\textrm{Type $\mathrm{IV}$ if and only if } p^2 \mid A, p^2 \| B &   \\ 
\end{array} \right.
\qquad
\left\{ \begin{array}{ccc}
\textrm{Type $\mathrm{II}^*$ if and only if } p^4 \mid A, p^5 \| B \\
\textrm{Type $\mathrm{III}^*$ if and only if }  p^3 \| A, p^5 \mid B \\
\textrm{Type $\mathrm{IV}^*$ if and only if } p^3 \mid A, p^4 \| B
\end{array} \right.
$$
Furthermore, for the fixed positive integer $m$, an elliptic curve 
$E$ has reduction type $\mathrm{I}_m^*$ at $p$ if and only if $p^2 \| A$ or $p^3 \|  B$,
and $v_p(4(A/p^2)^3 + 27(B/p^3)^2) = m$.
\end{lemma}
Finally for the local conditions $\mathrm{I}_m^*$, we need
\begin{lemma} \label{lemmaforIm}
For $m > 0$, we have
$$|\lcrc{ (\alpha, \beta) \in (\bZ/p^{m+6}\bZ)^2 : p^2 \| \alpha, p^3 \| \beta, p^{m+6} || \lbrb{4\alpha^3 + 27\beta^2}  }| = p^{m+5}(p-1)^2. $$
For $m = 0$, we have
$$|\lcrc{ (\alpha, \beta) \in (\bZ/p^{6}\bZ)^2 : p^2 \| \alpha, p^3 \| \beta, p^{6} \| \lbrb{4\alpha^3 + 27\beta^2}  }| = p^{6}(p-1). $$
\end{lemma}

\begin{proof}[Proof of Theorem \ref{main1}.]
We give a proof for $\cLC$ = good, $a_p(E), \mathrm{II}^*$, and $\mathrm{I}_m^*$.
Let $A:=\lcrc{(\alpha, \beta) \in (\bZ/p\bZ)^2 : 4\alpha^3 + 27\beta^2 \equiv 0 \pmod{p}}$. Now, the number of elliptic curves whose height is less than $X$, and have good reduction at $p \geq 5$ is
$$|\cE_{p}^{\mathrm{good}}(X)| = \sum_{(\alpha, \beta) \not\in A} | \cE_{p, \alpha, \beta}(X)|.$$
By Proposition \ref{42}, we have
\begin{eqnarray*}
|\cE_{p}^{\mathrm{good}}(X)| &=& \frac{(p^2-p)}{p^2}\frac{p^{10}}{p^{10}-1}\frac{4}{\zeta(10)}X^{\frac{5}{6}} +O((p-1)X^{\frac{1}{2}}).
\end{eqnarray*}

For $a$ in a Weil bound, 
\begin{equation*}
\frac{(p-1)H(a^2-4p)}{2p^2}\frac{p^{10}}{p^{10}-1}\frac{4}{\zeta(10)}X^{\frac{5}{6}}+O\left( (p-1)\frac{H(a^2-4p)}{2p}X^{\frac{1}{2}}\right),
\end{equation*}
by  (\ref{eqn:Cox}).
Similarly, by Lemma \ref{Tatealg}, an elliptic curve $E_{A,B}$ has type $\mathrm{II}^*$ at $p$ if and only if $(A,B)$ satisfies $p^4 \mid A, p^5 \|B$ modulo $p^7$.
Since there are $p^3(p^2-p)$ such pairs, we have 
$$|\cE_{p}^{\mathrm{II}^*}(X)| = 
\frac{p^3(p^2-p)}{p^{14}}\frac{p^{10}}{p^{10}-1}\frac{4}{\zeta(10)}X^{\frac{5}{6}} + O\lbrb{p^3(p^2-p)\frac{X^{\frac 12}}{p^7}}.$$

Now, $\mathrm{I}_m^*$-case directly followed by Lemma \ref{lemmaforIm}. More concretely,
\begin{equation*}
|\cE_p^{\mathrm{I}_m^* }(X)| = \sum_{\alpha, \beta} |\cE_{p^{m+6}, \alpha, \beta}(X)|
 = \frac{p^{m+5}(p-1)^2}{p^{2m+12}}\frac{p^{10} }{p^{10}-1} \frac{4}{\zeta(10)}X^{\frac56} + O\lbrb{\frac{p^{m+5}(p-1)^2}{p^{m+6}}X^{\frac12}}.
\end{equation*}
Other cases can be proven similarly.
\end{proof}

\subsection{With finitely many local conditions}
In this section, we give a proof of Theorem \ref{main4}.
Let $P = \lcrc{p_i^{m_i}}$ be a finite set of powers of primes such that $p_i \geq 5$, $I=I_{P}$ be a set of nonzero $(\alpha_i, \beta_i) \in (\bZ/p_i^{m_i}\bZ)^2$. We also define a convolution on $I$ by
$d*I := \lcrc{(d^4\alpha_i, d^6\beta_i)}. $
Let 
$$\cD_{P, I}(X):= \lcrc{(A, B) \in \cD(X) : (A,B) \equiv (\alpha_i, \beta_i) \pmod{p_i^{m_i}} \textrm{ for all } p_i^{m_i} \in P}, $$
and
$$\cM_{P, I}(X) := \lcrc{(A,B) \in \cD_{P, I}(X) : (A, B) \textrm{ satisfies }(M)}.$$
Then, the analogue of Lemma \ref{DMmodp} is as follows.
\begin{lemma} \label{MobCRT}
Assume that all pairs in $I$ satisfy $p_i^4 \nmid \alpha_i$ or $p_i^6 \nmid \beta_i$. Then,
$$ \cD_{P, I}(X) = \bigsqcup_{\substack{d \leq X^{\frac{1}{12}} \\ p_i \nmid d }} d*\cM_{P, d^{-1}*I}(d^{-12}X).$$
\end{lemma}
We also define 
$$\cE_{P, I}(X) := \lcrc{(A, B) \in \cM_{P, I}(X) : 4A^3 + 27B^2 \neq 0},$$
and $\cS_{P, I}(X)$ as the unique set satisfying $\cM_{P, I}(X) = \cE_{P, I}(X)\bigsqcup \cS_{P, I}(X)$. The bound of error term $\cS_{P, I}(X)$  can be obtained by following lemma.

\begin{lemma} \label{cSCRT}
For $P, I$, we have
$$|\cS_{P, I}(X)| \leq \frac{4^{|I|}}{\prod_i p_i^{m_i}}X^{\frac{1}{6}} + O(1). $$
\end{lemma}

\begin{proposition} \label{CRTcE}
Assume that all pairs in $I$ satisfy $p_i^4 \nmid \alpha_i$ or $p_i^6 \nmid \beta_i$, and $\prod p_i^{m_i} \leq X^{\frac{1}{3}}$. Then,
$$|\cE_{P, I}(X)| =\prod_i \lbrb{\frac{1}{p_i^{2m_i}}\frac{p_i^{10}}{p_i^{10}-1}}\frac{4}{\zeta(10)}X^{\frac{5}{6}}
+O\lbrb{\frac{X^{\frac{1}{2}}}{\prod p_i^{m_i}}}.$$
\end{proposition}
\begin{proof}
It can be derived the proof of Proposition \ref{42}, Lemma \ref{cSCRT}, and the Chinese remainder theorem.
\end{proof}

We recall some notations in Section \ref{Intro}. Let $\cLC_p$ be one of good, bad, split, non-split, an integer $a_p$  in the Weil's bound at $p$, or one of the Kodaira--N\'eron types $T$. We define 
$|\cLC_p|$ as a $\displaystyle \lim_{X \to \infty} \frac{|\cE_p^{\cLC_p}(X)|}{|\cE(X)|}$, which is $c_{\cLC_p}(p)$ in Theorem \ref{main1}.
For the finite set of local conditions $\cS = \{\cLC_{p_i}\}$, we define $|\cS| = \prod_i |\cLC_{p_i}|$.

Now we are ready to prove Theorem \ref{main4}.
We first assume that any $\cLC_p$ is neither bad nor additive.
%When $\cS = (\cLC_{p_i})$ is given, we simply write $m_i$ for $m_{\cLC_{p_i}}$.
Let $P = (p_i^{m_i})$, then in the proof of Theorem \ref{main1}, there are 
$$|\cLC_{p_i}| \lbrb{\frac{p_i^{10}-1}{p_i^{10}}}p_i^{2m_i}$$
pairs $(\alpha_{i,j}, \beta_{i,j})$ in $(\bZ/p_i^{m_i}\bZ)^2$ such that
$E_{A, B}$ satisfies $\cLC_{p_i}$ if and only if $(A,B) \equiv (\alpha_{i,j}, \beta_{i,j}) \pmod{p_i^{m_i}}$ for some $(\alpha_{i,j}, \beta_{i,j})$.
Let  $I_{p_i} = \lcrc{(\alpha_{i,j}, \beta_{i,j})}$ be the set of all the forementioned pairs.

Let $I = \prod I_{p_i}$. We denote an element of $I$ by $(\alpha, \beta)$
whose $I_{p_i}$ component is $(\alpha_{i,j}, \beta_{i,j})$, 
and define $(A, B) \equiv (\alpha, \beta) \pmod{P}$ if and only if $(A, B) \equiv (\alpha_{i,j}, \beta_{i,j}) \pmod{p_i^{m_i} }$.
By Chinese remainder theorem, there are 
$$\prod_i|\cLC_{p_i}| \lbrb{\frac{p_i^{10}-1}{p_i^{10}}}p_i^{2m_i}$$
elements $(\alpha, \beta)$ in $I$ such that $E_{A,B}$
satisfies all $\cLC_{p_i}$ if and only if $(A, B) \equiv (\alpha, \beta) \pmod{P}$ for some $(\alpha, \beta)$ in $I$. Together with Proposition \ref{CRTcE}, we have 
\begin{eqnarray*}
|\cE^{\cS}(X)| &=& \sum_I |\cE_{P, I}(X)|
= \sum_I\lbrb{\prod_i \lbrb{\frac{1}{p_i^{2m_i}}\frac{p_i^{10}}{p_i^{10}-1}}\frac{4}{\zeta(10)}X^{\frac{5}{6}}
+O\lbrb{\frac{X^{\frac{1}{2}}}{\prod p_i^{m_i}}}  } \\
&=&\lbrb{\prod_i|\cLC_{p_i}|}\frac{4}{\zeta(10)}X^{\frac{5}{6}} + O\lbrb{\prod_i|\cLC_{p_i}| \lbrb{\frac{p_i^{10}-1}{p_i^{10}}}p_i^{m_i}X^{\frac 12}} \\
&=&|\cS|\frac{4}{\zeta(10)}X^{\frac56}+ O\lbrb{\prod_i|\cLC_{p_i}| \lbrb{\frac{p_i^{10}-1}{p_i^{10}}}p_i^{m_i}X^{\frac 12}}.
\end{eqnarray*}
Hence we proved this theorem except when there is a local condition $\cLC_p$ which is either bad or additive.

We use an induction on the number of $\cLC_p$'s that are bad.
Let $\cS = \{\cLC_{p_i}\}$ be a finite set of local conditions such that there is an $\cLC_{p_n}$ which is bad,
and $\cS' = \cS - \{\cLC_{p_n} \}$. By induction hypothesis for $\cS' \cup \{ \cLC'_{p_n} \}$ where
$\cLC'_{p_n}$ is good, we have
$$ |\cE^{\cS' \cup \{ \cLC'_{p_n} \} }(X)| = |\cS'|\frac{p_n^2 - p_n}{p_n^2}\frac{p_n^{10}}{p^{10}_n-1}\frac{4}{\zeta(10)}X^{\frac56}+ O\lbrb{ \lbrb{\prod_{i=1}^{n-1}c_{p_i}} p_n X^{\frac 12}}.$$
Again by induction hypothesis for $\cS'$, we have
$$ |\cE^{\cS'}(X)| = |\cS'|\frac{4}{\zeta(10)}X^{\frac56}+ O\lbrb{ \lbrb{\prod_{i=1}^{n-1}c_{p_i}}  X^{\frac 12}}.$$
Since $ |\cE^{\cS'}(X)| =  |\cE^{\cS'\cup \{ \cLC'_{p_n} \}}(X)| + |\cE^{\cS}(X)|$, 
we can count elliptic curves with bad reduction conditions.

We can do the same thing for additive condition because  additive condition is the complement of the union of the conditions good, split, and non-split.

\section{The Frobenius Trace formula for elliptic curves} \label{trace}

Let $L(s,E)$ be the normalized elliptic $L$-function and for which we have 
\begin{align*}
L(s,E)=\sum_{n=1}^\infty \frac{\lambda_E(n)}{n^s},\qquad -\frac{L'}{L}(s,E)=\sum_{n=1}^\infty \frac{\widehat{a_E}(n)\Lambda(n)}{n^s}.
\end{align*}

Here is the Frobenius trace formula for elliptic curves. 
\begin{theorem}\label{traceF}[Frobenius Trace Formula for Elliptic Curves] Let $k$ be a fixed positive integer. 
Let $p_i$ be distinct primes $\geq 5$ such that $\prod_{i=1}^{k}p_i=O( X^{\frac{1}{3} - \epsilon})$ for a fixed positive $\epsilon > 0$. Assume $e_i=1$ or $2$, and $r_i$ is odd or $2$ if $e_i=1$, $r_i=1$ if $e_i=2$ for $i=1,\dots,k$. Then, 
\begin{align*} 
 \sum_{E \in \cE(X)} \widehat{a_E}(p_1^{e_1})^{r_1} \widehat{a_E}(p_2^{e_2})^{r_2} \cdots \widehat{a_E}(p_k^{e_k})^{r_k}&=c|\cE(X)| +O_k\left( 2^r\left( \prod_{i=1}^k p_i \right) X^{\frac 12} \right)+O_k\left( \left( \sum_{i=1}^k \frac{1}{p_i} \right) X^{\frac{5}{6} } \right)
\end{align*}
where
\begin{align*}
c=\left\{ \begin{array}{rl}
 0 & \mbox{ if $e_j=1$ and $r_j$ is odd for some $j$}, \\
 -1 & \mbox{if  $r_j=2$ for all $j$ with $e_j=1$, and the sum of $r_j$'s with $e_j=2$ is odd,} \\
 1 & \mbox{otherwise.}
 \end{array} \right. 
\end{align*}
and $r$ is either the smallest odd integer $r_i$ or $0$ if there is no odd $r_i$, and  the last error term exists only if $e_i=1$ and $r_i=2$ or $e_i=2$ for all $i$. 

Also, we have
\begin{align*} 
\sum_{E \in \cE(X)} \lambda_E(p_1^{e_1})^{r_1} \lambda_E(p_2^{e_2})^{r_2} \cdots \lambda_E(p_k^{e_k})^{r_k}&=d|\cE(X)|+O_k\left( 2^r\left( \prod_{i=1}^k p_i \right) X^{\frac 12} \right)+O_k\left( \left( \sum_{i=1}^k \frac{1}{p_i} \right) X^{\frac{5}{6} } \right)
\end{align*}
where 
\begin{align*}
d=\left\{ \begin{array}{rl}
 0 & \mbox{ if $e_j=2$ or $e_j=1$ and $r_j$ is odd for some $j$}, \\
 1 & \mbox{if  $e_i=1$ and $r_i=2$ for all $i$.}
 \end{array} \right.
\end{align*}
and the last error term exists only if $e_i=1$ and $r_i=2$ or $e_i=2$ for all $i$. 
\end{theorem}

We give a proof of Theorem \ref{traceF} for one prime and two primes. First, we consider
\begin{align*}
\sum_{E \in \cE(X)} \widehat{a}_E(p)^r
=  \sum_{|a| \leq \lfrf{2\sqrt{p}} }\sum_{\substack{E \in \cE(X) \\ a_p(E) = a} } \widehat{a}_E(p)^r +
\sum_{\substack{E \in \cE(X) \\ E \textrm{ mult red at } p} } \widehat{a}_E(p)^r +
\sum_{\substack{E \in \cE(X) \\ E \textrm{ addi red at } p} } \widehat{a}_E(p)^r
\end{align*} 
for an odd $r$.
By Theorem \ref{main1} for $|a|\leq \lfloor 2 \sqrt{p} \rfloor$, there are
$$
\frac{(p-1)H(a^2-4p)}{2p^2}\frac{p^{10}}{p^{10}-1}\frac{4}{\zeta(10)}X^{\frac{5}{6}}+O\left( (p-1)\frac{H(a^2-4p)}{2p}X^{\frac{1}{2}}\right)
$$
elliptic curves in $\cE(X)$ with $ \widehat{a}_E(p)=a/\sqrt{p}$.  Since $\sum_{|a|\leq \lfloor 2 \sqrt{p} \rfloor}a^rH(a^2-4p)=0$, only the contribution from the error term survives and it is at most $O(2^rpX^{\frac{1}{2}})$ because 
$$\frac{p-1}{2p}\sum_{|a|\leq \lfloor 2 \sqrt{p} \rfloor} \lbrb{\frac{|a|}{\sqrt{p}} }^r H(a^2-4p)\leq 2^{r-1}\frac{p-1}{p}\sum_{|a|\leq \lfloor 2 \sqrt{p} \rfloor} H(a^2-4p)=2^r(p-1).$$

If an elliptic curve $E$ has additive reduction at a prime $p$, then $\widehat{a}_E(p)=0$. So there is no contribution from the additive reduction case. 
If $E$ has split (resp. non-split) multiplicative reduction, then $\widehat{a}_E(p)=1/\sqrt{p}$ (resp. $\widehat{a}_E(p)=-1/\sqrt{p}$).  There are 
$$
\frac{p-1}{2p^2}\frac{p^{10}}{p^{10}-1}\frac{4}{\zeta(10)}X^{\frac{5}{6}}+O( X^{\frac{1}{2}} )
$$
elliptic curves in $\cE(X)$ with the split prime $p$, and also there are the same number of elliptic curves with the non-split prime $p$ up to error term $O(X^{\frac{1}{2}})$, by Theorem \ref{main1}. The contribution from multiplicative reduction is $O(X^{\frac{1}{2}}/p^{\frac{r}{2} })$, hence we have
\begin{align*}
\sum_{E \in \cE(X)} \widehat{a}_E(p)^r = O(2^rpX^{\frac{1}{2}}).
\end{align*}
The next case is 
\begin{align*}
\sum_{E \in \cE(X)} \widehat{a}_E(p)^2.
\end{align*} 
From the identity $\sum_{|a|\leq \lfloor 2 \sqrt{p} \rfloor}a^2H(a^2-4p)=2p^2-2$ which is an application of the Eichler--Selberg trace formula \cite{Bir},\footnote{There is a typo in \cite[Theorem 2]{Bir}.} the contribution for good reductions at $p$ is
\begin{align*}
\sum_{|a| \leq \lfrf{2\sqrt{p}} }\sum_{\substack{E \in \cE(X) \\ a_E(p) = a} } \widehat{a}_E(p)^2 = 
\left( 1+ O\left( \frac 1p \right) \right)\frac{4}{\zeta(10)}X^{\frac{5}{6}}+O( pX^{\frac{1}{2}}),
\end{align*}
and the contribution from the multiplicative reduction is 
\begin{align*}
\sum_{\substack{E \in \cE(X) \\ E \textrm{ mult red at } p} } \widehat{a}_E(p)^2 =
\frac{p-1}{p^3}\frac{p^{10}}{p^{10}-1}\frac{4}{\zeta(10)}X^{\frac{5}{6}}+O\left( \frac{X^{\frac{1}{2}}}{p}\right).
\end{align*}
Hence we have
\begin{align*}
\sum_{E \in \cE(X)} \widehat{a}_E(p)^2 = \frac{4}{\zeta(10)}X^{\frac{5}{6}} + O\left( \frac{X^{\frac{5}{6}}}{p} + pX^{\frac{1}{2}} \right).
\end{align*} 

The last case for a single prime $p$ is
\begin{align*}
\sum_{E \in \cE(X)} \widehat{a}_E(p^2).
\end{align*}
Note that $\widehat{a}_E(p^2)=\widehat{a}_E(p)^2-2$ when $E$ has good reduction at $p$, and $\widehat{a}_E(p^2)=\widehat{a}_E(p)^2$ when $E$ has bad reduction at $p$. Then, from the computation for the previous case, we can see easily that 
\begin{align*}
\sum_{E \in \cE(X)} \widehat{a}_E(p^2) = -\frac{4}{\zeta(10)}X^{\frac{5}{6}} + O\left( \frac{X^{\frac{5}{6}}}{p} + pX^{\frac{1}{2}} \right),
\end{align*} 
which is exactly the trace formula for one prime in Theorem \ref{traceF}.

Now, we consider the trace formula for two primes
\begin{align*}
\sum_{E \in \cE(X)} \widehat{a}_E(p_1^{e_1})^{r_1} \widehat{a}_E(p_2^{e_2})^{r_2}.
\end{align*}

Assume that $e_1=1$ and $r_1$ is odd. Fix a local condition for the second prime $p_2$. We vary the local conditions for the first prime $p_1$, and this gives $O(2^{r_1}p_1c_{p_2}X^{\frac{1}{2}})$. Then we sum up this term over all local conditions for $p_2$, the error term $O(2^{r_1}p_1p_2X^{\frac{1}{2}})$ follows. 

Now, we need to deal with the cases $e_i=1$ and $r_i=2$ or $e_i=2$ and $r_i=1$ for $i=1,2$.   First, we consider the cases $\cLC_{p_1}=a_1$, and $\cLC_{p_2}=a_2$ for $|a_1|\leq \lfrf{\sqrt{p_1}}$, $|a_2|\leq\lfrf{\sqrt{p_2}}$. Their contribution is, for example $r_1=r_2=2$, 
\begin{align*}
\left( \prod_{i=1}^2\frac{(p_i-1)p_i^{10}}{2p_i^3 (p_i^{10}-1)}\right)\frac{4}{\zeta(10)}X^{\frac 56}\cdot \sum_{|a_1| \leq \lfrf{\sqrt{p_1}},|a_2| \leq \lfrf{\sqrt{p_2}}}a_1^2H(a_1^2-4p_1)a_2^2H(a_2^2-4p_2) \\
+O\left(\sum_{\substack{ |a_1| \leq \lfrf{\sqrt{p_1}},|a_2| \leq \lfrf{\sqrt{p_2} } } } \frac{a_1^2H(a_1^2-4p_1)a_2^2H(a_2^2-4p_2)}{(2p_1)(2p_2)} X^{\frac 12} \right),
\end{align*}
which is, by the identity $\sum_{|a| \leq \lfrf{\sqrt{p}}}a^2H(a^2-4p)=2p^2-2$, 
\begin{align*}
\frac{4}{\zeta(10)}X^{\frac 56}+O\left(\left( \frac{1}{p_1}+ \frac{1}{p_2} \right)X^{\frac 56} + p_1p_2 X^{\frac 12} \right).
\end{align*}

Since the case that $p_1$ or $p_2$ has multiplicative reduction, using the trivial bound for $\widehat{a}_E(p_1^{e_1})^{r_1} \widehat{a}_E(p_2^{e_2})^{r_2}$, gives  $O\left(\left( \frac{1}{p_1^2}+ \frac{1}{p_2^2} \right)X^{\frac 56} \right)$, we verify the trace formula for $r_1=r_2=2$ case.  The other three cases can be handled similarly and we have
$$
\sum_{E \in \cE(X)} \widehat{a}_E(p_1^{e_1})^{r_1} \widehat{a}_E(p_2^{e_2})^{r_2}=(-1)^{r_1+r_2}\frac{4}{\zeta(10)}X^{\frac{5}{6}}+O\left(\left( \frac{1}{p_1} + \frac{1}{p_2} \right) X^{\frac{5}{6}} + p_1p_2 X^{\frac{1}{2}}\right).
$$
 For a general $k$ primes, we can prove the trace formula in the same way.

For $\lambda_E(p^e)$, we note that $\lambda_E(p)=\widehat{a}_E(p)$ and $\lambda_E(p^2)=\widehat{a}_E(p)^2-1$ if $E$ has good reduction at $p$ and 
 $\lambda_E(p^2)=\widehat{a}_E(p)^2$ otherwise. Hence, we can prove the trace formula for $\lambda_E(n)$ similarly. 

\section{The distribution of analytic ranks of elliptic curves}  \label{distribution}
From now on, assume that every elliptic $L$-function satisfies Generalized Riemann Hypothesis.  Let $\gamma_E$ denote the imaginary part of a non-trivial zero of $L(s,E)$. We index them using the natural order in real numbers:
\begin{align*}
\cdots \gamma_{E,-3} \leq \gamma_{E,-2} \leq \gamma_{E,-1} \leq \gamma_{E,0} \leq \gamma_{E,1} \leq \gamma_{E,2} \leq \gamma_{E,3} \cdots
\end{align*}
if analytic rank $r_E$ is odd,
\begin{align*}
\cdots \gamma_{E,-3} \leq \gamma_{E,-2} \leq \gamma_{E,-1} \leq 0  \leq \gamma_{E,1} \leq \gamma_{E,2} \leq \gamma_{E,3} \cdots
\end{align*}
otherwise.

In this section, first we obtain an upper bound on every $n$-th moment of analytic ranks of elliptic curves and obtain a bound on the proportion of elliptic curves with analytic rank $r_E \geq (1+C)9n$ for a positive constant $C$ and a positive integer $n$. For this purpose, we compute an $n$-level density with multiplicity. See \cite[Part VI]{Mil}.

For $\sigma_n= \frac{2}{9n}$, we choose the following test function. 
\begin{align*}
\widehat{\phi}_n(u)=\frac{1}{2}\left(\frac{1}{2}\sigma_n-\frac{1}{2}|u|\right)   \textrm{ for } |u|\leq \sigma_n, \qquad \textrm{and } \phi_n(x)=\frac{\sin^2(2\pi \frac 12 \sigma_n x)}{(2\pi x)^2}.
\end{align*}
Note that $\phi_n(0)=\frac{\sigma_n^2}{4}$ and $\widehat{\phi}_n(0)=\frac{\sigma_n}{4}$. We can understand the constant $\frac{2}{9}$ as the limit of our trace formula. Then, easily we can check 
\begin{align}  \label{estimate}
\int_\mathbb{R}|u|\widehat{\phi}_n(u)^2= \frac{1}{6}\phi_n(0)^2.
\end{align}

The $n$-level density with multiplicity is 
\begin{eqnarray*}
D_n^*(\cE, \Phi)=\frac{1}{|\cE(X)|}\sum_{E\in \cE(X)} \sum_{j_1,j_2,\dots, j_n}\phi_n \left( \gamma_{E, j_1} \frac{\log X}{2\pi}\right) \phi_n \left( \gamma_{E, j_2} \frac{\log X}{2\pi}\right)\cdots \phi_n \left( \gamma_{E, j_n} \frac{\log X}{2\pi}\right),
\end{eqnarray*}
where $\gamma_{E, j_k}$ is an imaginary part of $j_k$-th zero of $L(s, E)$.
Then, trivially we have
\begin{eqnarray} \label{moment-ineq}
\frac{1}{|\cE(X)|}\sum_{E\in \cE(X)} r_E^n \leq \frac{1}{\phi_n(0)^n}D_n^*(\cE, \Phi).
\end{eqnarray}

We show that for any $n$, $D_n^*(\cE, \Phi)$ has a closed expression. By Weil's explicit formula,
\begin{align*}
\sum_{j}\phi_n \left( \gamma_{E,j} \frac{\log X}{2 \pi} \right)&=\widehat{\phi}_n(0) \frac{\log N_E}{\log X}-2\sum_{k=1}^{2}\sum_p\frac{\widehat{a_E}(p^k) \log p}{p^{k/2} \log X} \widehat{\phi}_n\left( \frac{k \log p}{\log X}\right) + O\left( \frac{1}{\log X}\right)\\
& \leq \widehat{\phi}_n(0)-2\sum_{k=1}^{2}\sum_p\frac{\widehat{a_E}(p^k) \log p}{p^{k/2} \log X} \widehat{\phi}_n\left( \frac{k \log p}{\log X}\right) + O\left( \frac{1}{\log X}\right),
\end{align*}
where the inequality holds due to  $\log N_E/ \log X \leq 1 + O(1/\log X)$ by Ogg's formula \cite[\S 11]{Sil2}.

%We note that by Ogg's formula \cite[\S 11]{Sil2}, the conductor is always less or equal than the minimal discriminant. The discriminant $-16(4A^3 + 16B^2)$ of $E_{A,B}$ is not minimal at 2 or 3 in general. Therefore we have $N_E \leq 16(4A^3+27B^2)$, and by the height condition $|A|^3,  B^2 \leq X$, we have $\log N_E \leq \log X + 16\log 2$. Hence we have $\frac{\log N_E}{\log X} \leq 1 + O(1/\log X)$.

Hence, we have
\begin{align*}
&\frac{1}{|\cE(X)|}\sum_{E \in \cE(X)} \sum_{j_1,j_2,\dots, j_n}\phi_n \left( \gamma_{E,j_1} \frac{\log X}{2\pi}\right) \phi_n \left( \gamma_{E,j_2} \frac{\log X}{2\pi}\right)\cdots \phi_n \left( \gamma_{E, j_n} \frac{\log X}{2\pi}\right) \\
&\leq \frac{1}{|\cE(X)|}\sum_{E \in \cE(X)}\left( \widehat{\phi}_n(0) - \frac{2}{\log X} \sum_{m < X^{\frac{2}{9n} }} \frac{\widehat{a}_E(m)\log m}{\sqrt{m}} \widehat{\phi}_n \left( \frac{\log m}{\log X}\right)  + O\left(\frac{1}{\log X} \right) \right)^n ,
\end{align*}
where $m$'s are primes or squares of a prime. By the standard argument such as \cite[Lemma 2]{Rub} or \cite[Lemma 4.4]{CK15}, we can take the term $O(1/\log X)$
out of the bracket, and we have
\begin{align*}
D_n^*(\cE, \Phi) 
&\leq \frac{1}{|\cE(X)|}\sum_{S} \widehat{\phi}_n(0)^{|S^c|} \left( -\frac{2}{\log X} \right)^{|S|}\\ 
& \times   \sum_{m_{i_1}, m_{i_2}, \dots, m_{i_k}} \frac{\Lambda(m_{i_1})\Lambda(m_{i_2}) \cdots \Lambda(m_{i_{k}}) }{ \sqrt{m_{i_1}m_{i_2} \cdots m_{i_{k}}}} \widehat{\phi}_n \left( \frac{\log m_{i_1}}{\log X}\right) \cdots \widehat{\phi}_n \left( \frac{\log m_{i_k}}{\log X}\right)\\
& \times \sum_{ E \in \cE(X)} \widehat{a}_E(m_{i_1})\widehat{a}_E(m_{i_2}) \cdots \widehat{a}_E(m_{i_k}) + O\left(\frac{1}{\log X} \right),
\end{align*}
where $m_i$'s are primes or squares of a prime with $m_i \leq X^{\frac{2}{9n}}$ and $S=\{ i_i, i_2, \dots, i_k \}$ runs over every subset of  $\{1,2,3, \cdots, n\}$. In next two propositions, we show that only the terms $m_{i_1}m_{i_2}\cdots m_{i_k}=\square$ give the main term and the contribution of the other terms is $O(1/\log X)$.   

\begin{proposition} \label{non-square}
\begin{align*}
&\sum_{E \in \cE(X)} \sum_{m_{i_1} m_{i_2} \dots m_{i_k} \neq \square} \frac{\Lambda(m_{i_1})\cdots \Lambda(m_{i_{k}}) \widehat{a}_E(m_{i_1})\cdots \widehat{a}_E(m_{i_k})}{\sqrt{m_{i_1}m_{i_2} \cdots m_{i_{k}}}} \widehat{\phi}_n \left( \frac{\log m_{i_1}}{\log X}\right) \cdots \widehat{\phi}_n \left( \frac{\log m_{i_k}}{\log X}\right) \\
&\ll  |\cE(X)|. 
\end{align*}
\end{proposition}
\begin{proof}
Note that $\widehat{a}_E(m_{i_1})\widehat{a}_E(m_{i_2})\cdots \widehat{a}_E(m_{i_k})$ is of the form
$$
\widehat{a}_E(p_1)^{e_1}\widehat{a}_E(p_2)^{e_2}\cdots \widehat{a}_E(p_t)^{e_t}  \widehat{a}_E(q_1^2)^{l_1}  \widehat{a}_E(q_2^2)^{l_2} \cdots  \widehat{a}_E(q_s^2)^{l_s} , 
$$
with with $e_1+\cdots +e_t+ l_1  +\cdots+ l_s=k$. Here $p_1,p_2, \dots, p_t$ are distinct primes and $q_1,q_2,\dots, q_s$ are distinct primes, but some $q_j$ might be equal to some $p_i$. For a while we assume that the primes $p_1,\dots, p_t, q_1,\dots, q_s$ are all distinct. 

Since one of $e_i$'s is odd by our assumption, then by the Frobenius trace formula \ref{traceF},
\begin{align*}
\sum_{ E \in \cE(X)} \widehat{a}_E(p_1)^{e_1}\widehat{a}_E(p_2)^{e_2}\cdots \widehat{a}_E(p_t)^{e_t}  \widehat{a}_E(q_1^2)^{l_1}  \widehat{a}_E(q_2^2)^{l_2} \cdots  \widehat{a}_E(q_s^2)^{l_s} =O( p_1p_2\cdots p_t q_1 q_2 \cdots q_s X^{\frac{1}{2}}). 
\end{align*}

The contribution of this case in the worst situation is at most
\begin{align*}
\ll  X^{\frac{1}{2}}  \left( \sum_{p < X^{\frac{2}{9n}}}  p^{\frac{1}{2} } \log p \right)^k \ll  X^{\frac{1}{2}} ( X^{\frac{1}{3n}})^n  \ll X^{\frac{5}{6}}.
\end{align*}

Now, we assume that some $p_i$ is equal to some $q_j$. Since $\widehat{a}_E(q^2)^{l}=(\widehat{a}_E(q)^2-2)^l$ if $E$ has good reduction at $q$ and  $\widehat{a}_E(q^2)^{l}=\widehat{a}_E(q)^{2l}$ otherwise, still we can use the Frobenius trace formula. 
\end{proof}

\begin{proposition}\label{square} For a subset $S=\{ i_1, i_2, \dots, i_k \}$ of $\{1,2,\dots,n\}$,
\begin{align*}
\frac{1}{|\cE(X)|}\left(\frac{-2}{\log X} \right)^{|S|}\sum_{ E \in \cE(X)}\sum_{m_{i_1} m_{i_2} \dots m_{i_k} = \square}\left( \prod_{j=1}^{|S|} 
\frac{\Lambda(m_{i_j})\widehat{a}_E(m_{i_j}) }{\sqrt{m_{i_j}}} \widehat{\phi}_n \left(\frac{\log m_{i_j} }{\log X} \right)\right)\\
=\sum_{\substack{S_2 \subset S  \\ |S_2| \text{even}}} \left( \frac 12 \phi_n(0) \right)^{|S_2^c|} \left| S_2 \right|! \left(\int_\mathbb{R} |u| \widehat{\phi}_n(u)^2du \right)^{\frac{|S_2|}{2} } + O\left( \frac{1}{\log X} \right). 
\end{align*}
\end{proposition}
\begin{proof}
In this proof, we compute the double sum not considering the term $\frac{1}{|\cE(X)|}\left( \frac{-2}{\log X} \right)^{k}$. We show that  every contribution except one is $\ll X^{\frac 56} (\log X)^{k-1}$, hence they become the error term $O(1/ \log X)$  in the end. 

Note that $\widehat{a}_E(m_{i_1})\widehat{a}_E(m_{i_2})\cdots \widehat{a}_E(m_{i_k})$ is of the form
$$
\widehat{a}_E(p_1)^{e_1}\widehat{a}_E(p_2)^{e_2}\cdots \widehat{a}_E(p_t)^{e_t}  \widehat{a}_E(q_1^2)^{l_1}  \widehat{a}_E(q_2^2)^{l_2} \cdots  \widehat{a}_E(q_s^2)^{l_s} , 
$$
with with $e_1+\cdots +e_t+ l_1 + \cdots + l_s=k$ and $e_i$'s are all even. 
If $e_i \geq 4$ for some $i$ or $l_j \geq 2$ for some $j$, then by the trivial bound, this term is majorized by $X^{\frac{5}{6}}(\log X)^{k-1}$. Let $S_2$ be a subset of $S$ with even cardinality $2t$:
$$
S_2=\{ i_{a_1}, i_{a_2}, \cdots, i_{a_{2t-1}}, i_{a_{2t}} \}, \qquad S_2^c=\{i_{b_1},i_{b_2},\cdots,i_{b_s} \}.
$$
There are $(2t)!/2^t$ ways to pair up two elements in $S_2$. For example, we consider the following pairings.
\begin{align*}
(i_{a_1},i_{a_2}), (i_{a_3},i_{a_4}),  (i_{a_5},i_{a_6}), \cdots, (i_{a_{2t-1}},i_{a_{2t}}).
\end{align*}

This set of pairings corresponds the following sum 
\begin{align*}
 \sum_{ E \in \cE(X)} \widehat{a}_E(p_{i_{a_1}})^{2}\widehat{a}_E(p_{i_{a_3}})^{2}\cdots \widehat{a}_E(p_{i_{a_{2t-1}}})^{2}  \widehat{a}_E(q_{i_{b_1}}^2) \widehat{a}_E(q_{i_{b_2}}^2) \cdots  \widehat{a}_E(q_{i_{b_s}}^2)
\end{align*}
where $2t+s=k$. By the Frobenius trace formula, the above sum is
\begin{align*}
|\cE(X)|\cdot
\left\{ \begin{array}{clc} 1 &  \mbox{ if $s$ is even,} \\ -1 &  \mbox{ if $s$ is odd} \end{array}\right. + O\left(p_1\cdots p_t q_1 \cdots q_s X^{\frac{1}{2}}+\left( \frac{1}{p_1}+\cdots+ \frac{1}{p_t} + \frac{1}{q_1} + \cdots \frac{1}{q_s} \right)X^{\frac{5}{6}} \right).
\end{align*}

The contribution from the error term $O(p_1\cdots p_t q_1 \cdots q_s X^{\frac{1}{2}})$ is dominated by
$$
( X^{\frac{2}{9n}} \log X )^t ( X^{\frac{1}{9n}})^{s} X^{\frac{1}{2}} \ll X^{\frac{1}{2}+ \frac{2t+s}{9n}}(\log X)^t.
$$

The contribution from the error term $O\left(\left( \frac{1}{p_1}+\cdots+ \frac{1}{p_t} + \frac{1}{q_1} + \cdots \frac{1}{q_s} \right)X^{\frac{5}{6}}\right)$ is dominated by $X^{\frac{5}{6}}(\log X)^{k-1}$. The main term of the sum, after being divided by $|\cE(X)| \left( \frac{\log X}{-2}\right)^k$,   gives rise to 
\begin{align*} \prod_{i=1}^t \left( \left( \frac{-2}{\log X} \right)^2   \sum_{p} \frac{\log^2 p}{p} \widehat{\phi}_n \left(\frac{\log p}{\log X}\right)^2\right)
\times \prod_{j=1}^s\left(  \frac{2}{\log X} \sum_{q}\frac{\log q}{q} \widehat{\phi}_n \left( \frac{2 \log q }{\log X} \right) \right),
\end{align*}
which equals, by the prime number theorem, 
\begin{align*}
\left( 2^t   \prod_{i=1}^t \int_\mathbb{R} |u| \widehat{\phi}_n(u)^2du\right)\left( \left( \frac{1}{2}\right)^{s} \prod_{j=1}^s \int_\mathbb{R} \widehat{\phi}_n(u)du \right).
\end{align*}
Since there are $(2t)!/2^t$ ways to pair up two elements in $S_2$, the claim follows. 
\end{proof}

By Propositions \ref{non-square} and  \ref{square}, and (\ref{estimate}) we have the following inequality
\begin{align*}
D_n^*(\cE, \Phi) \leq \phi_n(0)^n \sum_{S}\left( \frac{1}{\sigma_n} \right)^{|S^c|}\sum_{\substack{S_2 \subset S  \\ |S_2| \text{even}}} \left( \frac 12  \right)^{|S_2^c|} \left| S_2 \right|! \left(\frac 16\right)^{|S_2|/2} + O\left( \frac{1}{\log X} \right), 
\end{align*}
and, by $(\ref{moment-ineq})$, Theorem \ref{n-th moment} follows.

\subsection{Proof of Theorem \ref{rank-dist}}
We choose the test function $\phi_{2n}(x)$. Then $\widehat{\phi}_{2n}(0)=\frac{1}{4}\sigma_{2n}$, and $\phi_{2n}(0)=\frac 14 \sigma_{2n}^2$.

By Weil's explicit formula, we have
$$r_E \phi_{2n}(0) \leq \widehat{\phi}_{2n}(0) -\frac{2}{\log X} \sum_{m_i}\frac{\widehat{a}_E(m_i)\Lambda(m_i)}{\sqrt{m_i}}\widehat{\phi}_{2n} \left( \frac{\log m_i}{\log X}\right) + O\left( \frac{1}{\log X}\right),$$
hence
\begin{align*}
r_E \leq \frac{1}{\sigma_{2n}} + \frac{4}{\sigma_{2n}^2}\left( -\frac{2}{\log X} \sum_{m_i}\frac{\widehat{a}_E(m_i)\Lambda(m_i)}{\sqrt{m_i}}\widehat{\phi}_{2n} \left( \frac{\log m_i}{\log X}\right) \right) + O\left( \frac{1}{\sigma_{2n}^2 \log X}\right).
\end{align*}

Now assume that $r_E \geq \frac{1+C}{\sigma_{2n}}$ with some positive constant $C$. Then, for sufficiently large $X$, 
\begin{align*}
-\frac{2}{\log X} \sum_{m_i}\frac{\widehat{a}_E(m_i)\Lambda(m_i)}{\sqrt{m_i}}\widehat{\phi}_{2n} \left( \frac{\log m_i}{\log X}\right) \geq \frac{C\sigma_{2n}}{4}.
\end{align*}

Therefore,
\begin{align*}
\left| \{ E \in \cE(X) | r_E \geq  \frac{1+C}{\sigma_{2n}} \} \right| & \left( \frac{C\sigma_{2n}}{4} \right)^{2n} 
 \leq \sum_{E \in \cE(X)}  \left( -\frac{2}{\log X} \sum_{m_i}\frac{\widehat{a}_E(m_i)\Lambda(m_i)}{\sqrt{m_i}}\widehat{\phi}_{2n} \left( \frac{\log m_i}{\log X}\right)  \right)^{2n} \\
& \leq \left( \frac{\sigma_{2n}^2}{4}\right)^{2n} \sum_{S_2 \subset \{1,2,3,\dots,2n \}} \left( \frac{1}{2} \right)^{|S_2^c|}|S_2|!\left( \frac 16 \right)^{|S_2|/2}|\cE(X)| + O\left(\frac{X^{\frac 56} }{\log X} \right),
\end{align*}
where the second inequality is justified by Propositions \ref{non-square}, \ref{square}, and finally we obtain
\begin{align*}
P\left(r_E \geq (1+C)\cdot 9n \right)  \leq \frac{\sum_{k=0}^{n}{ {2n} \choose {2k}}\left( \frac 12\right)^{2n-2k}(2k)!\left( \frac 16 \right)^k }{(C \cdot 9n)^{2n}}.
\end{align*}

\end{document}